\documentclass[11pt,reqno,a4paper]{amsart}
\usepackage[all, cmtip]{xy}
\usepackage{hyperref}
\usepackage{url}
\usepackage[utf8]{inputenc} 
\usepackage[T1]{fontenc}    
\usepackage{hyperref}       
\usepackage{url}   

\usepackage{booktabs}       
\usepackage{amsfonts}       
\usepackage{nicefrac}       
\usepackage{microtype}      
\usepackage[pdftex]{graphicx}
\usepackage{todonotes}
\usepackage{amsmath}

\usepackage{tikz-cd}
\usepackage{amsfonts}
\usepackage{mathrsfs}  
\usepackage{amssymb}
\usepackage{amsthm}
\usepackage{hyperref}
\usepackage[noadjust]{cite}
\usepackage{caption}
\usepackage{dutchcal}
\usepackage{bbm}
\usepackage{tikz}
\usetikzlibrary{decorations.pathreplacing}
\usepackage[T1]{fontenc}
\usepackage[utf8]{inputenc}
\usepackage{mathtools}
\usepackage{verbatim}

\usepackage[top=2.9cm, bottom=2.9cm, left=2.5cm, right=2.5cm, headsep=0.2in]{geometry}

\usepackage{fancyhdr}

\newtheorem{theorem}{Theorem}[section]

\newtheorem{corollary}[theorem]{Corollary}

\newtheorem{prop}[theorem]{Proposition}

\newtheorem{lemma}[theorem]{Lemma}

\newtheorem{definition}[theorem]{Definition}

\theoremstyle{remark}
	\newtheorem{remark}[theorem]{Remark}
	\newtheorem{exmp}[theorem]{Example}

\newcommand{\beq} {\begin{equation}}
\newcommand{\eeq} {\end{equation}}

\begin{document}

\title[Positivity on simple $G$-varieties]{Positivity on simple $G$-varieties}
\author{Praveen Kumar Roy, Pinakinath Saha}

\address{NMIMS MPSTME, Mumbai-400056, India}
\email{praveenkumar.roy@nmims.edu.in}

\address{Department of Mathematics, Indian Institute of Technology, Delhi, Hauz Khas, New
Delhi-110016, India.}
\email{pinakinath@iitd.ac.in}


\begin{abstract} 
Let $X$ be a normal projective variety equipped with an action of a semisimple algebraic group $G$, and assume that $X$ contains a unique closed orbit. Let $B$ be a Borel subgroup of $G$ and let $E$ be a $B$-equivariant vector bundle on $X$. In this article, we prove that $E$ is ample (respectively, nef) if and only if its restriction to the finite set of $B$-stable curves in $X$ is ample (respectively, nef).

Moreover, we compute the nef cone of the blow-up of a nonsingular simple $G$-projective variety $X$ at a unique $B$-fixed point $x^-$, referred to as the sink of $X$. As an application, when $X$ is nonsingular, we calculate the Seshadri constants of any ample line bundle (not necessarily $G$-equivariant) at $x^-$. In addition, we compute the Seshadri constants of $B$-equivariant vector bundles at $x^{-}$. 
\end{abstract}
\maketitle

\section{Introduction}
Let $X$ be a complex complete variety. A vector bundle $E$ on $X$ is said to be \textit{nef} if $\mathcal{O}_{\mathbb{P}(E)}(1)$ is a nef line bundle on the projective bundle $\mathbb{P}(E)$ over $X$.

Let $\Gamma$ be a connected solvable algebraic group, and  $X$ be a complete $\Gamma$-variety. Motivated by \cite{HMP}, we prove the following criteria for the nefness of any $\Gamma$-equivariant vector bundle $E$ on $X$. Our proof relies on a key result from \cite{fmss}, though it may also be obtained by following the argument presented in \cite{HMP}.

Our first result concerning the nefness for complete varieties is as follows:
\begin{theorem}\label{theorem: nef}
Let $\Gamma$ be a  connected solvable algebraic group, and  $X$ be a complete $\Gamma$-variety. A $\Gamma$-equivariant vector bundle $E$ on $X$ is nef if and only if the restriction $E|_{C}$ of $E$ to every $\Gamma$-stable closed curve $C$ on $X$ is nef.
\end{theorem}

Let $G$ be a connected semisimple algebraic group over $\mathbb{C}$, and let $B$ be a Borel subgroup of $G$ containing a maximal torus $T$. We consider a projective variety $X$ equipped with a $G$-action such that $X$ has a unique closed $G$-orbit. Such a variety is called a \textit{simple $G$-projective variety}. Examples include generalized flag varieties, wonderful compactifications of symmetric spaces, and simple spherical varieties.


Following \cite{bial}, \cite{konar}, and \cite{brion}, we recall some basic notation and properties of $G$-variety $X$. 

Let \( \lambda \in \mathcal{X}_*(T) \) be a cocharacter of $T$. The multiplicative group \( \mathbb{G}_m \) acts on \( X \) via the homomorphism \( \lambda: \mathbb{G}_m \to T \). Denote by \( X^\lambda \) the corresponding fixed point set under this action. For any point \( x \in X \), the orbit map \( \mathbb{G}_m \to X \), given by \( t \mapsto \lambda(t) \cdot x \), extends uniquely to a morphism \( \mathbb{P}^1 \to X \). This extension defines the limits
\[
\lim_{t \to 0} \lambda(t) \cdot x \quad \text{and} \quad \lim_{t \to \infty} \lambda(t) \cdot x,
\]
both of which lie in \( X^\lambda \). These two limits are distinct unless \( x \in X^\lambda \).
Recall also that \( X^\lambda = X^T \) for any \( \lambda \in \mathcal{X}^*(T) \) outside a finite union of proper subgroups; such a \( \lambda \) is called \emph{regular}.  
For any \( \lambda \in \mathcal{X}_*(T) \) and any closed subset \( Y \subseteq X^\lambda \), define
\[
X^-(Y) := \{ x \in X \mid \lim_{t \to \infty} \lambda(t) \cdot x \in Y \}.
\]
Then \( X^-(Y) \) is a locally closed, \( T \)-invariant subset of \( X \), and the map
\[
p^- : X^-(Y) \to Y, \quad x \mapsto \lim_{t \to \infty} \lambda(t) \cdot x
\]
is a surjective, affine, \( T \)-equivariant morphism. Moreover, \( X \) is the disjoint union of the subsets \( X^-(Y) \), where \( Y \) runs over all connected components of \( X^T \). These subsets are called the \emph{negative strata}.  
As a consequence, there exists a unique open negative stratum \( X^-(Y) \); the corresponding component \( Y = Y^- \) is irreducible and is called the \emph{sink} of \( X \) for the \( \lambda \)-action.

Any normal simple $G$-projective variety $X$ has a unique $B$-fixed point, namely $x^-$, called {\it sink} of $X$ (see \cite[Lemma 2]{brion}).

A vector bundle $E$ on $X$ is said to be ample if $\mathcal{O}_{\mathbb{P}(E)}(1)$ is an ample line bundle on the projective bundle $\mathbb{P}(E)$ over $X$. If $E$ is an ample vector bundle then it follows easily that the restriction of $E$ to any (closed) curve in $X$ is also ample. But the converse is not true in general, even when $E$ has rank one and $X$ has dimension two, as shown by an example of Mumford (see \cite[Example 10.6]{Har1} or \cite[Example 1.5.2]{Laz}). 

Nevertheless, it is natural to ask the following question in specific situations: if the restriction of $E$ to every curve in $X$ is ample, then under what conditions is $E$ itself ample?

This question has been studied in several situations. An affirmative answer is given when 
$E$ is a torus-equivariant vector bundle, either on a toric variety  (see \cite[Theorem 2.1]{HMP}) or on a 
generalized flag variety (see \cite{bhn}). A related question was studied for equivariant vector bundles on wonderful compactifications in \cite{BKN}, for Bott-Samelson-Demazure-Hansen varieties or wonderful compactifications of a symmetric varieties of minimal rank in \cite{bhks}, and for $G$-Bott-Samelson-Demazure-Hansen varieties in \cite{B-P}.

In this article, we explore a related question concerning ampleness for nonsingular simple $G$-projective varieties, as discussed above, and prove the following result:

\begin{theorem}\label{theorem: ample}
Let $X$ be a nonsingular simple $G$-projective variety, and let $E$ be a $B$-equivariant vector bundle on $X$. Then $E$ is ample if and only if its restriction to every $B$-stable closed curve $C$ on $X$ is ample.
\end{theorem}

We then shift our attention to Seshadri constants for nef vector bundles and line bundles on simple $G$-projective variety. Seshadri constants of a nef line bundle were defined by Demailly \cite{Dem}, motivated by the Seshadri's criterion for ampleness of a line bundle to measures the local positivity of $L$ at a point $x$. It is defined as follows:

\begin{definition}[Seshadri constant at a point]\label{ratio}
Let $X$ be a projective variety and $L$ be a nef line bundle on $X$. For $x\in X$, we define: 

\[
\varepsilon(X,L;x) := \inf\limits_{x\in C}\frac{L\cdot C}{{\rm mult}_xC},
\]	
\end{definition}
where infimum is running over all reduced and irreducible curves $C$ in $X$ passing through $x$. 
When the variety is clear from the context, we will write $\varepsilon(L;x)$ instead of $\varepsilon(X, L;x)$. 
Equivalently, Seshadri constant of $L$ at a point $x \in X$ can also be easily seen to be equal to: 
\[
 \varepsilon(L;x)=\sup\{\lambda \in \mathbb{R}:{\rm Bl}_{x}^*L-\lambda E \text{ is nef}\},
\]
where ${\rm Bl}_{x}:{\rm Bl}_{x} X\to X$ denotes the blow up of $X$ at the point $x$ and $E$ denotes the exceptional divisor corresponding to this blow up. 

By Seshadri's criterion for ampleness, \cite[Theorem 7.1]{Har1}, the line bundle $L$ is ample if and only if  $\varepsilon(X, L;x)>0$ for all $x\in X.$ Thus, the Seshadri constants of ample line bundles are positive real numbers.\\

The notion of Seshadri constant was later extended by Hacon \cite{Hac} to define it for nef vector bundles $E$ on a projective variety $X$. A version of Seshadri constants for ample vector bundles also appears implicitly in the works of Beltrametti, Scheider and Somesse (see \cite{Bel1, Bel2}). 
 
Computing and bounding these constants have been an active topic of research. Although most research in this area has focused on line bundles, especially on surfaces where the geometry is more tractable, there have also been notable contributions concerning higher-dimensional varieties. A non-exhaustive list of references for line bundles includes \cite{Lee, I, N, L2, K-P, B, GHS}, and for vector bundles, see \cite{BHM, BHMN, BKN, HMP, bhn}.

In this article, as an application of Theorem \ref{theorem: nef}, we compute the Seshadri constant of any $B$-equivariant nef vector bundle on $X$.

To state the main result of this paper, we first introduce some necessary notation. Let $B^-$ be the Borel subgroup of $G$ opposite to $B$ determined by $T$. Let $x^-\in X$ be the $B$-fixed point of $X$ and let $X^-\subseteq X$ be its unique $B^-$-stable open affine neighborhood. Let $D_1,\ldots, D_r$ be the irreducible components of $X \setminus X^-$ (see \cite[Theorem 1]{brion}). 

We now state our main results on Seshadri constants, first in the case of line bundles:
\begin{theorem}\label{SC:line-bundle}
 Let $X$ be a nonsingular simple $G$-projective variety and $L = \sum_{i} a_iD_i$ be an ample line bundle on $X$, not necessarily $G$-equivariant. Then we have 
 \[
 \varepsilon(L, x^-)= \min_i\{a_i\}.
 \]
\end{theorem}
Next, we state the corresponding result for vector bundles. Before doing so, we recall the following result from \cite{brion}: for any nonsingular simple $G$-projective variety $X$, by \cite[Lemma 6]{brion} and \cite[Theorem 2 (4)]{brion} there are only finitely many $B$-stable curves on $X$, and each of them is isomorphic to $\mathbb{P}^1$.
\begin{theorem}\label{Seshadri-constant:vect-bundle}
Let $X$ be a nonsingular simple $G$-projective variety. Let $E$ be a $B$-equivariant nef vector bundle on $X$ of rank $n$, and let $x^{-}$ be the sink of $X$. Then
 \[\varepsilon(E,x^{-})= \min_{i,C}\{ a_{i}(C)\}\]
where the minimum is taken over all finitely many $B$-stable curves $C$ on $X$ passing through $x^{-}$ and integers $\{a_{1}(C), \ldots , a_n(C)\}$ such that 
\[
E|_{C} \simeq \mathcal{O}_{C}(a_{1}(C))\oplus \cdots \oplus \mathcal{O}_{C}(a_{n}(C)).
\]
\end{theorem}

The paper is organised as follows: In section \ref{Eq:nefVB}, we prove Theorem \ref{theorem: nef} and Theorem \ref{theorem: ample}. In section \ref{SC}, we first describe the nef cone of the blow-up of a nonsingular simple $G$-projective vareity at the sink (see Lemma \ref{lemma:blow-up}), and use this description to prove Theorem \ref{SC:line-bundle}. We also prove Theorem \ref{Seshadri-constant:vect-bundle} in this section.

\section{Equivariant nef vector bundle}\label{Eq:nefVB}

In this section, we provide a general criterion for the nefness of $\Gamma$-equivariant vector bundle on a complete $\Gamma$-variety $X$, where $\Gamma$ is a connected solvable group. Before proceeding, we first recall and restate the following theorem from \cite{fmss} for the reader’s convenience.
\begin{theorem}[Theorem 1, \cite{fmss}]\label{fmss}
    Let $\Gamma$ be a connected solvable linear algebraic group acting on a complete variety $X$. Then the canonical homomorphism 
    \[
    A_k^{\Gamma}(X)\to A_k(X)
    \]
    is an isomorphism for all $1\le k \le \dim X$, where $A_k(X)$ denotes the Chow group of $k$-cycles modulo rational equivalence, and $A_k^{\Gamma}(X)$ denotes the $\Gamma$-stable $k$-cycles modulo rational equivalence.
\end{theorem}

\begin{proof}[Proof of Theorem \ref{theorem: nef}]
If $E$ is nef, then clearly $E|_{C}$ is nef for every closed curve $C$ on $X.$
	
	Conversely, assume that ${E}$ is a $\Gamma$-equivariant vector bundle on $X$ such that its restriction $E|_{C}$ to every $\Gamma$-stable curve $C$ in $X$ is nef.
	
	Let 
 \[
 \pi : \mathbb{P}(E)\to X
 \] 
 be the projective bundle over $X$ parametrizing the hyperplanes in the fibers of $E.$ Since $E$ is $\Gamma$-equivariant, $\mathbb{P}(E)$ is an $\Gamma$-variety and $\pi$ is $\Gamma$-equivariant. The
	tautological relative ample line bundle over $\mathbb{P}(E)$ will be denoted by $\mathcal{O}_{\mathbb{P}(E)}(1).$ To prove
	that $E$ is nef, we need to show that $\mathcal{O}_{\mathbb{P}(E)}(1)|_{\widetilde{C}}$ is nef for every closed curve $\widetilde{C}$ in  $\mathbb{P}(E).$
	Note that if $\pi(\widetilde{C})$ is a point, then $\mathcal{O}_{\mathbb{P}(E)}(1)|_{\widetilde{C}}$ is ample, because $\mathcal{O}_{\mathbb{P}(E)}(1)$ is relatively ample.
	
	Now assume that $\widetilde{C}$ is a curve in $\mathbb{P}(E)$ such that $\pi(\widetilde{C})$ is not a point.
	Since $\mathbb{P}(E)$ is a complete $\Gamma$-variety, by Theorem \ref{fmss} the cycle of $\widetilde{C}$ is rationally equivalent to a $\Gamma$-stable cycle $\widetilde{C}'.$ Further, since $\pi$ is $\Gamma$-equivariant, $\pi(\widetilde{C}')$ is a $\Gamma$-stable curve in $X.$ Therefore by the assumption, the restriction $E|_{\pi(\widetilde{C}')}$ is nef. Hence we have \[{\rm degree}(\mathcal{O}_{\mathbb{P}(E)}(1)|_{\widetilde{C}'})\ge 0.\] Since $\widetilde{C}$ is rationally equivalent to $\widetilde{C}'$, we have   
	\[{\rm degree}(\mathcal{O}_{\mathbb{P}(E)}(1)|_{\widetilde{C}})\ge 0.\] This proves that $E$ is nef if $E|_{C}$ is nef for every closed $\Gamma$-stable curve $C$ in $X.$	
\end{proof}
We now prove the following result, which will be used in the proof of Theorem \ref{Seshadri-constant:vect-bundle}.
\begin{prop} \label{Blow up: Proposition 1.2}
Let \( x \) be a nonsingular \( \Gamma \)-fixed point of \( X \), and let
\[
\mathrm{Bl}_{x} : \mathrm{Bl}_{x} X \to X
\]
denote the blow-up of \( X \) at \( x \). Then the following statements hold:
\begin{enumerate}
    \item The action of \( \Gamma \) lifts to \( \mathrm{Bl}_{x} X \).

    \item Let \( F \) be a \( \Gamma \)-equivariant vector bundle on \( \mathrm{Bl}_{x} X \). Then \( F \) is nef if and only if its restriction \( F|_{\widetilde{C}} \) to every \( \Gamma \)-stable closed curve \( \widetilde{C} \subset \mathrm{Bl}_{x} X \) is nef.

    \item Let \( W_x \) denote the exceptional divisor of the blow-up \( \mathrm{Bl}_{x} \). Let \( E \) be a \( \Gamma \)-equivariant vector bundle on \( X \). Then, for every integer \( m \), the vector bundle
    \[
    \mathrm{Bl}_{x}^{*}E \otimes \mathcal{O}_{\mathrm{Bl}_{x} X}(m W_{x})
    \]
    is \( \Gamma \)-equivariant.
\end{enumerate}
\end{prop}
\begin{proof}
Since $x$ is a nonsingular $\Gamma$-fixed point, the group $\Gamma$ acts on the tangent space $T_{x}X.$ 
Note that the exceptional divisor of the blow-up ${\rm Bl}_{x}$ is isomorphic to $\mathbb{P}(T_{x}(X))$, so $\Gamma$ acts on the exceptional divisor of the blow-up via its linear action on $T_{x}X.$  As ${\rm Bl}_{x}$ is an isomorphism outside the exceptional divisor, it follows that the $\Gamma$-action lifts to the entire space ${\rm Bl}_{x} X$. In fact, we have the following commutative diagram, \[
\xymatrix{
 \Gamma \times {\rm Bl}_{x} X \ar@{-->}[r]^{} \ar[d]_{(Id, {\rm Bl}_x)} & {\rm Bl}_{x} X \ar[d]_{{\rm Bl}_x} \\
\Gamma \times X \ar[r]^{} & X}
\]
where the map $\Gamma \times {\rm Bl}_{x} X \longrightarrow {\rm Bl}_{x} X$ is uniquely determined by the universal property of the blow-up \cite[Theorem 7.14]{Har}. 

This completes the proof of part (1).

Since the action of $\Gamma$ lifts to ${\rm Bl}_{x} X$, the proof of (2) follows directly from Theorem \ref{theorem: nef}.

Finally, since $\Gamma$ acts on the exceptional divisor $W_{x}$ of ${\rm Bl}_{x},$ we conclude that $\mathcal{O}_{{\rm Bl}_{x} X}(W_{x})^{\otimes m}$ is a $\Gamma$-equivariant line bundle for every integer $m.$ 

Hence, if $E$ is a $\Gamma$-equivariant vector bundle on $X$ then 
\[
{\rm Bl}_{x}^{*}E \otimes \mathcal{O}_{{\rm Bl}_{x} X}(W_x)^{\otimes m}
\]
is a $\Gamma$-equivariant vector bundle on ${\rm Bl}_{x} X$ for every integer $m.$

This completes the proof of part (3).
\end{proof}

Following the argument presented in \cite[Theorem 2.1]{HMP}, we obtain the following result.

\begin{proof}[Proof of Theorem \ref{theorem: ample}]
The restriction of an ample vector bundle to a closed subvariety is always ample. Conversely, assume that the restriction of $E$ to every $B$-stable curve is ample.
Since $X$ is a nonsingular simple $G$-projective variety, it follows from \cite[Lemma 6]{brion} that $X$ contains only finitely many $B$-stable curves, and each of them is isomorphic to $\mathbb{P}^1$.
Let us fix a $B$-equivariant ample line bundle $L$ on $X$ and choose an integer $m$ greater than $L\cdot C$ for every $B$-stable curve $C$ in $X.$ The restriction of ${\rm Sym}^{m}(E)\otimes L^{-1}$ to each $B$-stable curve $C$ is nef, and hence ${\rm Sym}^{m}(E)\otimes L^{-1}$ is nef by Theorem \ref{theorem: nef}. It follows that ${\rm Sym}^{m}(E)$ is ample, and therefore $E$ is ample as well (see \cite[Proposition 6.2.11 and Theorem 6.1.15]{Laz}). 
\end{proof}

Now we consider flag varieties, which are the simplest examples of nonsingular simple $G$-projective varieties. These varieties are rational homogeneous projective varieties under the actions of semisimple groups. To state our result in this direction, we first introduce the following notation.

For a representation $V$ of $B$, the associated vector bundle on $G/B$ is denoted by $\mathbb{V}$, and it is defined as 
\[
\mathbb{V} = G \times^{B} V = G\times V/\sim
\]
where the equivalence relation is given by the $B$-action: ($g, v)\sim (gb, b^{-1}v)$ for $b\in B$, $g \in G$ and $v \in V.$ This construction describes a $G$-equivariant vector bundle over the flag variety $G/B$, where $G$ acts on the left. Moreover, every $G$-equivariant vector bundle on $G/B$ arises in this way.

For $G/B$, the $B$-stable irreducible closed curves are precisely the Schubert varieties of dimension one, commonly referred to as Schubert curves. 

For these varieties, our Theorem \ref{theorem: nef} and Theorem \ref{theorem: ample} imply the following:

\begin{corollary}\label{cor:ample/flag}
A $G$-equivariant vector bundle $\mathbb{V}$ on $G/B$ is nef (respectively, ample) if and only if its restriction to every Schubert curve is nef (respectively, ample).
\end{corollary}
\begin{proof}
Follows from Theorem \ref{theorem: nef} and Theorem \ref {theorem: ample}.
\end{proof}
\begin{remark}
By \cite{bhn}, the $G$-equivariant vector bundle $E$ on 
$G/B$ is ample (respectively, nef) if and only if its restriction to $T$-stable curves on $G/B$ are ample (respectively, nef). However, by Corollary \ref{cor:ample/flag}, it suffices to check this condition only on $B$-stable curves, which is a smaller subset of $T$-stable curves.
\end{remark}

\section{Seshadri constant on a simple $G$-projective variety}\label{SC}
In this section, we focus on the positivity of any nef line bundle (not necessarily equivariant) on a nonsingular simple $G$-projective variety $X$. Further, we focus on the positivity of a $B$-equivariant vector bundle $E$ on such a nonsingular simple $G$-projective variety $X$. 

We first recall the following crucial result from \cite[Theorem 1]{brion} that determines the nef cone of divisors of $X$ as follows:

 \begin{theorem}[Nef cone]\label{divisors}
Let $X$ be a nonsingular simple $G$-projective variety. Let $B^-$ be the Borel subgroup of $G$ opposite to $B$ determined by $T$, let $x^-\in X$ be the $B$-fixed point, $X^-\subseteq X$ its unique $B^-$-stable open affine neighborhood, and $D_1,\ldots, D_r$ be the irreducible components of $X \setminus X^-$. Then the following hold:
\begin{enumerate}
\item 
$D_1,\ldots,D_r$ are globally generated Cartier divisors. Their
linear equivalence classes form a basis of the Picard group of $X$. 
\item
Every ample divisor on $X$ is linearly equivalent to a
unique linear combination of $D_1,\ldots,D_r$ with positive
integer coefficients. 
\item
Every nef divisor on $X$ is linearly equivalent to a
unique linear combination of $D_1,\ldots,D_r$ with non-negative
integer coefficients. 
\end{enumerate}
\end{theorem}

We rewrite the Theorem \cite[Theorem 2]{brion} determining the cone of effective 1-cycle, as follows:
\begin{theorem}[Cone of effective 1-cycles]\label{curves}
    Let $X$ be a nonsingular simple $G$-projective variety and let $D_i$ be as in Theorem \ref{divisors}. Then
    \begin{enumerate}
    	\item The sink of each $D_i$ is a unique point $x_i^-$, isolated in $X^T$. 
        \item The cone ${\rm NE}(X)$ is generated by the classes of $C_1,\ldots,C_r$, where $C_i=\overline{B\cdot x_i^-}$.
        \item Every $D_i$ intersects $C_i$ transversally and $D_i \cdot C_j = \delta_{ij}$. Moreover, all $C_i$ are isomorphic to $\mathbb{P}^1$, and their classes form a basis of the group $N_1(X)$.
    \end{enumerate}
\end{theorem}

Let $X$ be a nonsingular simple $G$-projective variety and let 
\[
{\rm Bl}_{x^-} : {\rm Bl}_{x^-} X \longrightarrow X
\]
be the blow-up map at the sink $x^-$ of $X$. Denote $D := \sum\limits_{i=1}^rD_i$. Let $H = {\rm Bl}_{x^-}^*D$. Note that the strict transform $\tilde{C_i}$ of the curve $C_i$ is given by the class ${\rm Bl}_{x^-}^{*} C_i - l$, 
where $l$ denotes the class of a line in the exceptional divisor $E$ (see \cite[Corollary 6.7.1]{Fulton}).
Before proving our main theorem, we prove the following. 
\begin{lemma}\label{lemma:blow-up}
	Let $X$ be a nonsingular simple $G$-projective variety and let 
	\[
	{\rm Bl}_{x^-} : {\rm Bl}_{x^-} X \longrightarrow X
	\]
	be the blow-up map at the sink $x^-$. Then the line bundles
	\[
	{\rm Bl}_{x^-}^*D_1, {\rm Bl}_{x^-}^*D_2, \ldots, {\rm Bl}_{x^-}^*D_r,
	\]
	and $H - E$ are nef. Moreover,
	\[
	 {\rm Nef}({\rm Bl}_{x^-} X) = \langle{\rm Bl}_{x^-}^*D_1, {\rm Bl}_{x^-}^*D_2, \ldots, {\rm Bl}_{x^-}^*D_r, H- E \rangle,
	\]
	and its dual, the Mori cone of curves is generated by classes $\tilde{C_i}$ and $l$:
	\[
	 {\rm NE}({\rm Bl}_{x^-} X) = \langle\tilde{C_1} , \tilde{C_2} , \ldots, \tilde{C_r} , l \rangle,
	\]
	where $\tilde{C_i} := {\rm Bl}_{x^-}^{*} C_i - l$ with $l$ being the class of a line in $E$.
\end{lemma}
\begin{proof}
Since $D_i$ is a globally generated line bundle on $X$ (see Theorem \ref{divisors}), ${\rm Bl}_{x^-}^*D_i$ is a globally generated line bundle on ${\rm Bl}_{x^-} X$; in particular, it is nef. Now we prove that $H - E$ is a nef line bundle. Then we show that all the nef line bundles on ${\rm Bl}_{x^-} X$ are non-negative linear combination of ${\rm Bl}_{x^-}^*D_i$ and $H-E$. 

 Note that $D_i$ is linearly equivalent to $n_{w_0}D_i$, where $n_{w_0}$ is a representative of the longest element $w_0$ of the Weyl group $W$ in the normalizer of $T$ in $G$.
 
 Note that $D_i$ are $B^-$-stable, so $D$ is $B^-$-stable. Therefore, the divisor $D$ is linearly equivalent to the $B$-stable divisor $n_{w_0}D$. Let $D' := n_{w_0}D$ and $H' =  {\rm Bl}_{x^-}^*D'$. Then $H$ is linearly equivalent to the $B$-stable divisor $H'$. 
 
 Since $n_{w_0}D_i$ are $B$-stable, they contain a $B$-fixed point. Further, since $x^-$ is the unique $B$-fixed point of $X$, $n_{w_0}D_i$ passes through the point $x^-$ for each $i$. Therefore, 
 \[
 {\rm Bl}_{x^-}^*n_{w_0}D_i - E
 \]
 is effective; in particular $H' -E$ is effective, as 
 \[
 H' -E = (H' - rE) + (r-1)E.
 \] 
 Recall that the blow up map ${\rm Bl}_{x^-}$ is $B$-equivariant. Since $D'$ is $B$-equivariant, $H'$ is $B$-equivariant. By Proposition \ref{Blow up: Proposition 1.2}, note that, $E$ is $B$-equivariant. Therefore, by Theorem \ref{Blow up: Proposition 1.2} (3), $H' - E$ is $B$-equivariant. By Theorem \ref{theorem: nef}, it suffices to show that $H' -E$ intersects non-negatively with any $B$-stable curve on ${\rm Bl}_{x^-}X$. 
 
 Let $C \subset {\rm Bl}_{x^-}X$ be an irreducible $B$-stable curve. If $C \subset E$, then 
 \[
 (H' -E) \cdot C = -E\cdot C > 0.
 \]
 Otherwise, since the blow up map ${\rm Bl}_{x^-}$ is $B$-equivariant, and ${\rm Bl}_{x^-}(C)$ is an irreducible $B$-stable curve in $X$, it follows from Theorem \ref{curves} that
 \[
 {\rm Bl}_{x^-}(C) = \sum_{i=1}^s a_iC_i
 \]
 for some $1\leq s \leq r$ and $a_i \geq 1$ for all $1\leq i \leq s$. 
Note that $C_i$ is isomorphic to $\mathbb{P}^1$ by Theorem \ref{curves} (3). Note that the class of $C$ is numerically equivalent to the class of $\sum_{i=1}^s a_i\tilde{C_i}$. Therefore, 
 \begin{eqnarray*}
 (H - E)\cdot C &=& (H' -E)\cdot \sum_{i=1}^s a_i ({\rm Bl}_{x^-}^{*} C_i - l) \\
 &=& \sum_{i=1}^s a_i (H' -E)\cdot ({\rm Bl}_{x^-}^{*} C_i - l) \\
 &=& \sum_{i=1}^s a_i(D'\cdot C_i + E\cdot l) \\
 &=&  \sum_{i=1}^s a_i (D'\cdot C_i - 1) \\
 &=& \sum_{i=1}^s a_i(D\cdot C_i -1) = 0, \quad \text{(by Theorem \ref{curves} (2))}
 \end{eqnarray*}
hence $H - E$ is nef.
	
Now, let $L'$ be a nef line bundle on ${\rm Bl}_{x^-} X$. Since the Picard group of ${\rm Bl}_{x^-} X$ is generated by ${\rm Bl}_{x^-}^*D_i$ and $E$, we can write 
	\[
	L' = \sum\limits_{i=1}^n b_i{\rm Bl}_{x^-}^*D_i - cE.
	\] 
	Since $L'$ is nef, we have $c = L' \cdot l \geq 0$ and $b_j - c = L' \cdot  \tilde{C_j} \geq 0$. 
	Further, since $H- E$ is nef, we can write $L'$ as
	\[
	 L' = c(H - E) + \sum\limits_{i=1}^n (b_j - c){\rm Bl}_{x^-}^*D_i.
	\]
	This completes the proof of the first part of the lemma. 
	
	For the second part, note that $\tilde{C_j}$ and $l$ form dual classes for the generating nef divisors in Nef(${\rm Bl}_{x^-} X$), and hence generate the Mori cone of curves NE(${\rm Bl}_{x^-} X$).
\end{proof}

\begin{remark}
If $C_1, . . . , C_r$ are the only $B$-stable irreducible curves, then the above lemma follows easily from Theorems \ref{fmss}, \ref{divisors}, and \ref{curves} as follows. Any $B$-stable curve on $Bl_{x^-} X$ is the strict transform of some $\tilde{C_i}$ or a curve in the exceptional divisor $E$. Hence ${\rm NE}({\rm Bl}_{x^-} X) = \langle\tilde{C_1} , \tilde{C_2} , \ldots, \tilde{C_r} , l \rangle$,
by Theorem \ref{fmss}. Then the dual cone ${\rm Nef} ({\rm Bl}_{x^-} X)$ can be computed directly using Theorem \ref{curves}.
\end{remark}

\begin{remark}\label{Rem:nefcone}
Following the notation of above lemma, we observe that for each $j$, the line bundle ${\rm Bl}_{x^-}^*D_j - E$ is not nef on ${\rm Bl}_{x^-} X$. Indeed, its intersection with $\tilde{C_i}$ is $-1$ for $i \neq j$.
\end{remark}

We now present the proof of our main theorem on the Seshadri constants of an ample line bundle on a nonsingular simple $G$-projective variety $X$. 

\begin{proof}[Proof of Theorem \ref{SC:line-bundle}]
Consider the blow-up map ${\rm Bl}_{x^-}:{\rm Bl}_{x^-} X \to X$ of $X$ at the sink $x^-\in X$ with exceptional divisor $E$. Then, the Seshadri constant of $L$ at $x^-$ is given by 
\[
\varepsilon(L;x^-)=\sup\{\lambda: {\rm Bl}_{x^-}^*L-\lambda E \text{ is nef}\}.
\]
By Lemma \ref{lemma:blow-up}, the line bundle ${\rm Bl}_{x^-}^*L- \min_i\{a_i\} E$ is nef. Therefore, we have
\[
\varepsilon(L;x^-) \geq \min_i\{a_i\}.
\]
On the other hand, for each $i$, the Seshadri ratio of $L$ corresponding to the curve $C_i$ passing through $x^-$ is $\frac{L\cdot C_i}{1} = a_i$. 
Hence, we obtain 
\[
\varepsilon(L;x^-) = \min_i\{a_i\}.
\]
This completes the proof of the theorem. 
\end{proof}

\begin{remark}
For any line bundle $L$ on a nonsingular simple $G$-projective variety $X$, 
\[
\varepsilon(L;x^{-}) = \varepsilon(L; g\cdot x^{-})
\]
for all $g \in G.$
\end{remark}

\begin{exmp}
Let \( X = \mathbb{P}^n \) and \( G = \mathrm{PGL}(n+1) \). Let \( x_0, \ldots, x_n \) be the homogeneous coordinates on \( \mathbb{P}^n \). Under the natural action of \( G \) on \( X \), the variety \( X \) is a nonsingular simple \( G \)-projective variety.

In this setting, the sink of \( X \) is the point \( x^{-} = [1:0:\cdots:0] \), and the hyperplane divisor is given by \( H = \{ x_0 = 0 \} \). Consider the curve \( C = \{ x_2 = \cdots = x_n = 0 \} \), which is isomorphic to \( \mathbb{P}^1 \).

Let \( D = mH \) for some \( m > 0 \). Then, by Theorem~\ref{SC:line-bundle}, we have:
\[
\varepsilon(D; x^{-}) = m.
\]

Note that for any point \( x \in \mathbb{P}^n \), there exists \( g \in G \) such that \( x = g \cdot x^{-} \). Therefore, it follows that
\[
\varepsilon(D; x^{-}) = \varepsilon(D; g \cdot x^{-}) = \varepsilon(D; x).
\]
\end{exmp}

\begin{exmp}
Let \( G \) be a semisimple algebraic group of rank \( n \), and let \( B \) be a Borel subgroup of \( G \). Let \( X = G/B \) be the flag variety associated with \( G \). Then \( X \) is a nonsingular simple \( G \)-projective variety with unique sink \( x^{-} = [B] \).

For each \( i = 1, 2, \ldots, n \), let \( s_i \) be the simple reflection corresponding to the simple root \( \alpha_i \), and define:
\[
D_i = \overline{B^{-}s_i B}/B, \quad C_i = \overline{Bs_iB}/B.
\]
Note that \( D_i \cdot C_j = \delta_{ij} \).

Let \( L = \sum\limits_{i=1}^n a_i D_i \) be an ample line bundle. Then, by Theorem~\ref{SC:line-bundle}, we have:
\[
\varepsilon(L; x^{-}) = \min\limits_{i} \{a_i\}.
\]

Moreover, for any point \( x \in G/B \), there exists \( g \in G \) such that \( x = g \cdot x^{-} \). Since the Seshadri constant is invariant under the action of \( G \), it follows that
\[
\varepsilon(L; x) = \varepsilon(L; x^{-}).
\]
\end{exmp}

\begin{exmp}
Let \( G = \mathrm{PGL}(2) \), and let \( B \) be the image of the upper triangular invertible matrices in \( \mathrm{SL}(2) \) under the quotient map \( \mathrm{SL}(2) \to \mathrm{PGL}(2) \). Consider the natural action of \( G \) on \( \mathbb{P}^1 \). 

Now consider the diagonal action of \( G \) on \( (\mathbb{P}^1)^3 := \mathbb{P}^1 \times \mathbb{P}^1 \times \mathbb{P}^1 \). Let \( p_1, p_2, p_3 \in (\mathbb{P}^1)^3 \) be pairwise distinct points, and let 
\[
X = \overline{G \cdot (p_1, p_2, p_3)} \subset (\mathbb{P}^1)^3
\]
be the closure of the \( G \)-orbit. Then \( X = (\mathbb{P}^1)^3 \) is a simple \( G \)-projective variety with sink \( \infty^3 := \infty \times \infty \times \infty \).

The irreducible divisors \( D_i \) and the curves \( C_i \) (as in Theorems~\ref{divisors} and \ref{curves}) are given by
\[
D_i = \left( (\mathbb{P}^1)^{i-1} \times \infty \times (\mathbb{P}^1)^{3-i} \right) \cap X, \quad 
C_i = \left( \infty^{i-1} \times \mathbb{P}^1 \times \infty^{3-i} \right) \cap X.
\]

Let \( L = \sum\limits_{i=1}^3 a_i D_i \) be any ample line bundle on \( X \). Then, it follows from Theorem~\ref{SC:line-bundle} that 
\[
\varepsilon(L; \infty^3) = \min\{a_i\}.
\]
In particular, the Seshadri constant of the anti-canonical line bundle $-K_X = 2(D_1 + \cdots + D_r)$ is
\[
\varepsilon(-K_X; \infty^3) = 2.
\]
\end{exmp}

\subsection{Seshadri constants of vector bundles}

Seshadri constants for nef vector bundles were defined in \cite{Hac}. We recall the definition below:

Let \( E \) be a nef vector bundle on a projective variety \( X \), and let \( x \in X \). Let \( \mathrm{Bl}_x: \mathrm{Bl}_x X \to X \) denote the blow-up of \( X \) at \( x \). We have the following commutative diagram:
\[
\xymatrix{
	\mathbb{P}(\mathrm{Bl}_x^*E)\ar[r]^{\widetilde{\mathrm{Bl}_x}} \ar[d]_{q} &
	\mathbb{P}(E) \ar[d]_{p} \\
	\mathrm{Bl}_x X \ar[r]^{\mathrm{Bl}_x} &
	X
}
\]
Let \( \xi = \mathcal{O}_{\mathbb{P}(\mathrm{Bl}_x^*E)}(1) \) denote the tautological line bundle on \( \mathbb{P}(\mathrm{Bl}_x^*E) \), and let \( Y_x = p^{-1}(x) \), \( Z_x = \widetilde{\mathrm{Bl}_x}^{-1}(Y_x) \). 

The Seshadri constant of \( E \) at \( x \) is defined as:
\[
\varepsilon(E; x) := \sup\left\{\lambda \in \mathbb{R} ~\middle|~ \xi - \lambda \cdot \mathcal{O}_{\mathbb{P}(\mathrm{Bl}_x^*E)}(Z_x) \text{ is nef} \right\}.
\]

Using the notation above, we now present the proof of our main theorem on the Seshadri constants of nef vector bundles on a nonsingular simple \( G \)-projective variety \( X \):

\begin{proof}[Proof of Theorem \ref{Seshadri-constant:vect-bundle}]
Let \( W_{x^{-}} \) denote the exceptional divisor of the blow-up
\[
\mathrm{Bl}_{x^{-}}: \mathrm{Bl}_{x^-} X \to X
\]
at the point \( x^{-} \in X \). By Proposition \ref{Blow up: Proposition 1.2}, the blow-up \( \mathrm{Bl}_{x^-} X \) inherits a \( B \)-action, and the map is \( B \)-equivariant.

By definition, the Seshadri constant of \( E \) at \( x^{-} \) is:
\[
\varepsilon(E; x^{-}) = \sup\left\{ \lambda \in \mathbb{R} ~\middle|~ \xi - \lambda \cdot q^*\left(\mathcal{O}_{\mathrm{Bl}_{x^-} X}(W_{x^{-}})\right) \text{ is nef} \right\}.
\]
By Theorem \ref{theorem: nef}, to prove that \( \xi - \lambda \cdot q^*(\mathcal{O}_{\mathrm{Bl}_{x^-} X}(W_{x^{-}})) \) is nef, it suffices to verify:
\[
\left( \xi - \lambda \cdot q^*(\mathcal{O}_{\mathrm{Bl}_{x^-} X}(W_{x^{-}})) \right) \cdot D \geq 0
\]
for every \( B \)-stable curve \( D \subset \mathbb{P}(\mathrm{Bl}_{x^{-}}^* E) \). 

For such a curve \( D \), set \( \widetilde{C} := q(D) \). Then \( D \subset \mathbb{P}(\mathrm{Bl}_{x^{-}}^* E|_{\widetilde{C}}) \). Consider the diagram:
\[
\xymatrix{
\mathbb{P}(\mathrm{Bl}_{x^-}^*E|_{\widetilde{C}})\ar@{^{(}->}[r] \ar[d]_{q_1} & \mathbb{P}(\mathrm{Bl}_{x^-}^*E)\ar[r]^{\widetilde{\mathrm{Bl}_{x^-}}} \ar[d]_{q} &
\mathbb{P}(E) \ar[d]_{p} \\
\widetilde{C} \ar@{^{(}->}[r] & \mathrm{Bl}_{x^-} X \ar[r]^{\mathrm{Bl}_{x^-}} & X
}
\]

Then:
\begin{align*}
(\xi - \lambda \cdot q^*(\mathcal{O}_{\mathrm{Bl}_{x^-} X}(W_{x^{-}}))) \cdot D 
&= \left( \mathcal{O}_{\mathbb{P}(\mathrm{Bl}_{x^{-}}^*E|_{\widetilde{C}})}(1) - \lambda \cdot q_1^*(\mathcal{O}_{\mathrm{Bl}_{x^-} X}(W_{x^{-}})|_{\widetilde{C}}) \right) \cdot D \\
&= \deg(\mathrm{Bl}_{x^-}^* E|_{\widetilde{C}}) - \lambda \cdot (\mathcal{O}_{\mathrm{Bl}_{x^-} X}(W_{x^{-}}) \cdot \widetilde{C}).
\end{align*}

So \( \xi - \lambda \cdot q^*(\mathcal{O}_{\mathrm{Bl}_{x^-} X}(W_{x^{-}})) \) is nef if and only if
\[
M_{\lambda, \widetilde{C}} := \mathcal{O}_{\mathbb{P}(\mathrm{Bl}_{x^{-}}^* E|_{\widetilde{C}})}(1) - \lambda \cdot q_1^*(\mathcal{O}_{\mathrm{Bl}_{x^-} X}(W_{x^{-}})|_{\widetilde{C}})
\]
is nef over $\mathbb{P}(\mathrm{Bl}_{x^-}^*E|_{\widetilde{C}})$ for every \( B \)-stable curve \( \widetilde{C} \subset \mathrm{Bl}_{x^-} X \). 

Assume $\widetilde{C} \subset W_{x^{-}}$. Since $W_{x^{-}}$ is the exceptional divisor of ${\rm Bl}_{x^-} X$, we have
\[
\mathcal{O}_{{\rm Bl}_{x^-} X}(W_{x^{-}})|_{W_{x^{-}}} = \mathcal{O}_{W_{x^{-}}}(-1).
\]
Therefore, $M_{\lambda, \widetilde{C}}$ is nef for every $\lambda\ge 0.$

Now assume \( \widetilde{C} \not\subset W_{x^{-}} \), and let \( C := \mathrm{Bl}_{x^{-}}(\widetilde{C}) \), a \( B \)-stable curve in \( X \). Since \( x^{-} \in C \) and \( C \simeq \mathbb{P}^1 \), we have:
\[
\mathcal{O}_{\mathrm{Bl}_{x^-} X}(W_{x^{-}}) \cdot \widetilde{C} = 1.
\]
The bundle \( E|_C \) splits as:
\[
E|_C = \mathcal{O}_C(a_1(C)) \oplus \cdots \oplus \mathcal{O}_C(a_n(C)),
\]
for some integers $a_1(C),\ldots ,a_n(C)$ (see \cite[Th\'eor\`eme 1.1]{Grothendieck}). Then  $M_{\lambda, \widetilde{C}}$ is nef if and only if $\mathcal{O}_{\widetilde{C}}(a_1(C)-\lambda)\oplus \cdots \oplus \mathcal{O}_{\widetilde{C}}(a_n(C)-\lambda)$ is nef. 
This direct sum is nef if and only if $\lambda \le {\rm min}\{a_1(C),\ldots, a_{n}(C)\}.$ Since there are only finitely many $B$-stables curves on $X$ (see \cite[Lemma 6]{brion}), it follows that:
\[
\varepsilon(E; x^{-}) = \min_{i, C} \{ a_i(C) \}.
\]
\end{proof}

\begin{remark}
Let \( X \) be a nonsingular simple \( G \)-projective variety. Let \( x \) be any \( T \)-fixed point such that the number of \( T \)-stable curves through \( x \) is finite. Let \( E \) be a \( T \)-equivariant nef vector bundle on \( X \) of rank \( n \). Then, imitating the proof of Theorem~\ref{Seshadri-constant:vect-bundle}, we conclude:
\[
\varepsilon(E; x) = \min_{i, C} \{ a_i(C) \},
\]
where the minimum is over all \( T \)-stable curves \( C \) through \( x \), and integers \( a_1(C), \ldots, a_n(C) \) such that
\[
E|_C \simeq \mathcal{O}_C(a_1(C)) \oplus \cdots \oplus \mathcal{O}_C(a_n(C)).
\]
\end{remark}

\begin{remark}
Similar computations of Seshadri constants have been carried out for equivariant vector bundles on toric varieties in \cite{bhks, bhn, BKN, HMP}.
\end{remark}

\section*{Acknowledgements}
The authors would like to thank their respective institutions for providing a supportive and productive research environment. The second named author acknowledges the National Board for Higher Mathematics (NBHM) Post Doctoral Fellowship with Ref. Number 0203/21(5)/2022-R\&D-II/10342.

\section*{data availability statement}
This manuscript has no associated data.

\section*{Conflict of interest}
On behalf of all authors, the corresponding author states that there is no conflict of interest.

\bibliographystyle{plain}

\end{document}